\documentclass[11pt]{article}

\usepackage{amsfonts,amsmath,amssymb,amsthm,graphicx,mathtools,hyperref,verbatim}
\usepackage[margin=1in]{geometry}

\newtheorem*{zl*}{Zolotarev's Lemma}
\newtheorem*{qr*}{Law of Quadratic Reciprocity}

\newtheorem{theorem}{Theorem}

\newtheorem{proposition}{Proposition}

\theoremstyle{definition}

\newtheorem*{exercise*}{Exercise}

\newcommand\R{{\mathbb R}}

\title{The sesquicentennial of the prime number $2^{127} - 1$}

\author{Matthew Baker}

\begin{document}

\maketitle

The year 2026 marks the $150^{\rm th}$ anniversary of a remarkable achievement in mathematics: in 1876, the French mathematician {\'E}douard Lucas 
% (1842--1891) 
showed that the 39-digit number $M_{127} := 2^{127} -1$ is prime; see \cite{Lucas1876b}.
Explicitly, we have
\[
M_{127} = 170141183460469231731687303715884105727.
% 170,141,183,460,469,231,731,687,303,715,884,105,727.
\]
This stood for 75 years as the largest known prime, and it remains the largest prime number discovered and certified without the aid of a mechanical device.
Moreover, as we shall see, the method Lucas used is rather astonishing.

The methods which Lucas pioneered are still used today by computers.
The current record for the largest number known to be prime  is $2^{136,279,841} - 1$, which has 41,024,320 decimal digits and was 
found on October 12, 2024 as part of the Great Internet Mersenne Prime Search (GIMPS) \cite{GIMPS}. The proof that this number is prime uses the \emph{Lucas--Lehmer test}, one of the most beautiful algorithms in number theory, about which we will have more to say.

Lucas, who was born in 1842 in the French town of Amiens, was a strikingly original mathematician. According to the book ``\'Edouard Lucas and Primality Testing'' by Hugh C. Williams \cite{WilliamsLucasBook}, Lucas was ``the first individual to show that primality testing could be performed, at least on certain forms of numbers, without recourse to the laborious and tedious process of trial division.'' Aside from his work on primality testing, he also made notable contributions to recreational mathematics.
For example, he invented the Tower of Hanoi puzzle, which he marketed under the nickname {\em N. Claus de Siam} (an anagram of Lucas d'Amiens), and he published the first description of the Dots and Boxes game.

Lucas died in quite an unusual way: a waiter at a banquet accidentally dropped the platter he was carrying and a piece of broken plate cut Lucas on the cheek. Lucas developed a severe skin inflammation and died a few days later, at the age of 49.

\begin{figure}[h]
\centering
\includegraphics[scale=1.5]{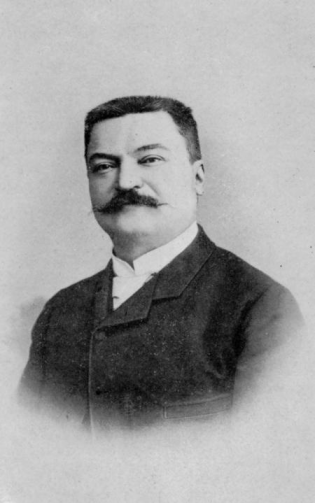}
\caption{Édouard Lucas (1842–1891).}
\end{figure}

\subsection*{Lucas's proof that $M_{127}$ is prime}

One of Lucas’ marvelous discoveries was that the celebrated Fibonacci sequence $(F_n)$ defined by $F_0 = 0, F_1 = 1$, and $F_{n+1} = F_n + F_{n-1}$ for $n \geq 1$ has a natural “companion sequence” $(L_n)$, 
defined by 
\[
L_0 = 2, L_1 = 1, \textrm{ and } L_{n+1} = L_n + L_{n-1} \textrm{ for } n \geq 1.   
\]
The relationship between the sequences $(F_n)$ and $(L_n)$ is, roughly speaking, like the relationship between sine and cosine. For example, we have 
\begin{equation}\label{eq:F2n}
F_{2n}=F_nL_n,
\end{equation}
similar to the formula $\sin(2x) = 2 \sin(x) \cos(x)$.

Lucas was especially interested in primitive prime divisors of these sequences. Given an integer sequence $(a_k)_{k \geq 0}$, a prime $q$ is said to be a \emph{primitive prime divisor} of $a_n$ if 
\[
q \mid a_n \textrm{ but } q \nmid a_m \textrm{ for all } 0 \leq m < n.
\]
One of Lucas's interesting theorems is that a primitive prime divisor $q \neq 5$ of the $n^{\rm th}$ Fibonacci number $F_n$ must be congruent to $\pm 1$ modulo $n$.

Lucas discovered the sequence $(L_n)$ which now bears his name in connection with primality testing, and he showed that $M_{127}$ is prime using roughly the following argument, translated into modern terms; see \cite{Lucas1876a}.\footnote{Lucas’s proofs were not fully rigorous by modern standards, but as Williams explains in \cite{WilliamsLucasBook} the gaps are minor and easily repaired. Williams also adds, ``It should be kept in mind that during the brief time during which Lucas was developing his seminal work on primality testing [1876-1878]\ldots he wrote at least 70 papers on many other subjects.''}
% \footnote{By modern standards, Lucas did not provide rigorous proofs of his theorems. However, Williams writes: ``Lucas has been criticized for his many errors and omissions, particularly by Carmichael; however, there can be no doubt that for the most part Lucas understood very well the basic principles behind his results. It is true that his statements and proofs thereof often leave something to be desired; but, on the other hand, these defects are often very easily repaired and do not require the machinery that Carmichael thought necessary. Indeed, it is the author’s opinion that Carmichael’s work, while putting Lucas on a much more solid mathematical footing, tended to muddy the waters rather than clarify them\ldots Certainly, his presentation lacks the charm and infectious enthusiasm that so characterizes Lucas’ work. It should be kept in mind that during the brief time during which Lucas was developing his seminal work on primality testing [1876-1878]\ldots he wrote at least 70 papers on many other subjects. Considering this immense outpouring of activity during such a brief time period, it is easy to forgive him for the few, easily correctable deficiencies that we have mentioned.''}

\begin{proposition}
Let $p\ge 3$ be prime and set $M_p=2^p-1$. If $M_p \mid L_{2^{p-1}}$, then $M_p$ is prime.
\end{proposition}

\begin{proof} (Sketch)
Let $q$ be any prime divisor of $M_p$. Since $M_p\mid L_{2^{p-1}}$ by hypothesis, we have
$q \mid L_{2^{p-1}}$.
Using \eqref{eq:F2n} with $n=2^{p-1}$,
\[
F_{2^p}=F_{2^{p-1}}\,L_{2^{p-1}},
\]
so $q\mid F_{2^p}$.

If $q\mid F_{2^{p-1}}$, then $q$ would divide both
$F_{2^{p-1}}$ and $L_{2^{p-1}}$. However, Lucas knew that
$\gcd(F_n,L_n) \mid 2$ for all $n$, which contradicts the fact that $q$ is odd.
Thus $q\nmid F_{2^{p-1}}$.

Lucas also knew that $\gcd(F_m,F_n) = F_{\gcd(m,n)}$ for all $m,n$, 
and combined with $q\nmid F_{2^{p-1}}$, this implies that
$q$ is a primitive prime divisor of $F_{2^p}$.
Thus 
\begin{equation}
\label{qeqpm1}
q\equiv \pm1\pmod{2^p}.
\end{equation}

Since $q\mid 2^p-1$, we have $q\le 2^p-1<2^p$,
and the congruence \eqref{qeqpm1} forces $q\ge 2^p-1$.
Therefore 
\[
q = 2^p - 1 = M_p.
\qedhere
\]
\end{proof}

% First of all, by examining the factorizations of various Fibonacci numbers, Lucas noticed that if $p$ is a prime divisor of $F_n$ such that $p \nmid F_m$ for all proper divisors $m$ of $n$ (such a prime is called a \emph{primitive prime divisor} of $F_n$), 
% then $p \equiv \pm 1 \pmod{n}$.  (For example, 17 is a primitive prime divisor of $F_9 = 34$ and a non-primitive prime divisor of $F_{27} = 196418$.)  Formulas \eqref{eq:Lucas1} and \eqref{eq:Lucas2} readily imply that if $p$ is an odd prime and 
% $p \mid L_{2^n}$, then $p$ is a primitive prime divisor of $F_{2^{n+1}}$ and hence $p \equiv \pm 1 \pmod{2^{n+1}}$.  

Thus, to prove that $M_{127}$ is prime, it suffices to verify that $M_{127} \mid L_{2^{126}} $. And this is what Lucas actually did --- 
but how, given that $L_{2^{126}}$ exceeds $10^{10^{37}}$?

Lucas noticed \cite{Lucas1876b} that if we set $r_k = L_{2^k}$ then $r_0 = 1, r_1 = 3$, and $r_{k+1} = r_k^2 - 2$ for $k \geq 1$; thus he needed to show that $r_{126} \equiv 0 \pmod{M_{127}}$.  This requires performing about 120 squaring operations and 120 divisions with numbers of up to 39 digits.  As Williams writes in Section 3.2: ``This is a lot of work by hand, but Lucas’ solution to this problem is very characteristic of him; he made the performance of this tedious arithmetic into something like a game.  He used a $127 \times 127$ chessboard to effect the computations.'' 

Lucas would encode the squaring procedure in binary on the chessboard, with pawns in squares representing 1 and empty squares representing 0.  He would take pawns away and move them around the board according to a specific procedure until the only pawns remaining were in the first row, and then the first row would give the least residue of $r_k^2$ modulo $M_{127}$ in binary.  In the end, Lucas proved in this way that $M_{127}$ is prime without ever writing anything down!

\subsection*{The Lucas--Lehmer test}

Although there are general-purpose algorithms for certifying primality, for example the celebrated polynomial-time test due to Agrawal, Kayal, and Saxena, such algorithms are much slower than special-purpose algorithms designed for numbers of a specific form.  
All of the largest known prime numbers are \emph{Mersenne primes}, meaning that they have the form $n = 2^p - 1$ for some odd prime $p$.\footnote{If $2^m - 1$ is prime them $m$ must itself be prime, so requiring that $p$ be prime is not actually a restriction. Indeed, if $m=ab$ then $2^m - 1=(2^a)^b-1=(2^a-1)(2^{a(b-1)}+2^{a(b-2)}+\cdots+2^a+1)$ and the factors on the right-hand side exceed $1$ when $a,b>1$. Since $2^{11}-1=2047=23\cdot 89$, primality of $p$ is necessary but not sufficient for $2^p - 1$ to be prime.}
As of the time of this writing, there are 52 known Mersenne primes, and it is unknown whether there are infinitely many such primes.

The reason that Mersenne primes can be certified as prime in a particularly efficient way is because of the \emph{Lucas--Lehmer test}, which is based on the pioneering ideas of Lucas sketched above, as refined by D.H. Lehmer.\footnote{Derrick Henry Lehmer (1905--1991) was a long-time professor of mathematics at U.C. Berkeley.  In 1950, Lehmer was one of 31 University of California faculty members fired for refusing to sign a loyalty oath during the McCarthy era; in 1952, the California Supreme Court declared the oath unconstitutional, and Lehmer returned to Berkeley. Lehmer built a number of fascinating mechanical ``sieve'' machines designed to factor large numbers. His first sieve, constructed in the students' workshop at U.C. Berkeley while Lehmer was still an undergraduate, used 19 separately looped bicycle chains which hung in loops from 10-tooth sprockets on a common shaft driven by a motor. \cite{LehmerHistoryPioneers}}
In what follows we give a ``modern'' exposition of the Lucas--Lehmer test combining observations of Michael Rosen \cite{Rosen1988}, J.W. Bruce \cite{Bruce1993}, and Greg Kuperberg (personal communication).\footnote{I have chosen to emphasize the role of the Chebyshev polynomial $T(x) = x^2 -2$ and its iterates in the following argument; this is not the standard approach to the proof but I personally find it enlightening.}

Let $T(x)=x^2-2$ and define the orbit of $4$ under iteration of $T$ by
\[
s_0=4, \qquad s_{n+1}=T(s_n)=s_n^2-2 \quad (n\ge0).
\]
Thus $(s_n) = T^{\circ n}(4)$ is the sequence
\[
4,\;14,\;194,\;37634,\;\dots
\]

Let $p$ be an odd prime and set $q:=2^p-1$.
We note that $q\equiv7\pmod{24}$, since induction on $k$ shows that
\[
2^k\equiv 8 \pmod{24}\quad\text{for all odd $k\ge3$.}
\]

\begin{theorem}[Lucas--Lehmer]
\label{thm:maintheorem}
For $p\ge 3$ prime,
\[
q =2^p-1 \text{ is prime}
\quad\Longleftrightarrow\quad
s_{p-2}\equiv 0 \pmod q.
\]
\end{theorem}

The key identity behind our proof will be the following special property of the Chebyshev polynomial $T(x) = x^2-2$, which is a linearly rescaled version of the trigonometric duplication formula $\cos(2\theta)=2\cos^2(\theta)-1$:
\begin{equation} \label{eq:ChebyshevProperty}
T(u+u^{-1})=(u+u^{-1})^2-2=u^2+u^{-2}.
\end{equation}
Thus, under the transformation
$x=u+u^{-1}$,
iteration of $T$ corresponds to the squaring map
$u \mapsto u^2$.

If $F$ is a field and $\alpha \in F$ satisfies $\alpha+\alpha^{-1}=4$, then by induction, for all $k \geq 0$ we have:
\begin{equation}\label{eq:alphaformula}
T^{\circ k}(4)=\alpha^{2^k}+\alpha^{-2^k}.
\end{equation}
For $F=\R$, solving $\alpha+\alpha^{-1}=4$ gives $\alpha=2+\sqrt3$.

For later use, note that in any field $F$,
\begin{equation} \label{eq:u=1}
u + u^{-1} = 2 \quad\Longleftrightarrow\quad u=1.
\end{equation}

\begin{proof}[Proof of Theorem \ref{thm:maintheorem}]
($\Rightarrow$)
Suppose $q$ is prime.
Since $q\equiv7\pmod{12}$, the Law of Quadratic Reciprocity implies
that $3$ is not a quadratic residue modulo $q$.

Let
\[
K=\mathbb F_q(\sqrt3),
\]
be the unique quadratic extension of $\mathbb F_q$, where $\sqrt{3}$ denotes a fixed square root of $3$ in $\mathbb F_q$.
Thus $K$ is a finite field with $q^2$ elements whose elements have the form $a+b\sqrt3$ with $a,b \in \mathbb F_q$.

Denote the nontrivial automorphism of $K$ sending $a+b\sqrt{3}$ to $a-b\sqrt{3}$ by $\alpha\mapsto\bar\alpha$.
Since the Frobenius map $\alpha \mapsto \alpha^q$ is also a non-trivial automorphism of $K$ and the Galois group of $K/\mathbb F_q$ has order 2, it follows that $\bar\alpha = \alpha^q$ for all $\alpha \in K$.

Consider the ``circle'' group
\[
C = \{\alpha\in K^\times : \alpha\bar\alpha=1\} 
= \{\alpha\in K^\times : \alpha^{q+1}=1\} . 
\]
Since $K^\times$ is a cyclic group of order $(q+1)(q-1)$, it follows that $C$ is cyclic of order $q+1 = 2^p$.

Also, $2+\sqrt3\in C$ since
\[
(2+\sqrt3)(2-\sqrt3)=1.
\]

{\bf Claim.} $2+\sqrt3$ generates $C$.

\smallskip

To see this, note that since $|C|$ is a power of $2$, the element $2+\sqrt3$
generates $C$ if and only if it is not a square in $C$.

Suppose for the sake of contradiction that $2+\sqrt3=c^2$ for some $c\in C$.
Let $b=c+\bar c\in\mathbb F_q$.
Using the Chebyshev identity \eqref{eq:ChebyshevProperty},
\[
b^2-2
=
c^2+\bar c^2
=
(2+\sqrt3)+(2-\sqrt3)
=
4.
\]
Thus $b^2=6$ in $\mathbb F_q$.
Since $q\equiv \pm 1 \pmod{8}$, the Law of Quadratic Reciprocity implies
that $2$ is a quadratic residue modulo $q$.
Since $3$ is a quadratic non-residue modulo $q$, $6=2\cdot 3$ must also be a non-residue.
This contradiction proves the claim.
% shows that $2+\sqrt3$ generates $C$.

Hence
\[
(2+\sqrt3)^{2^p}=1
\quad\text{but}\quad
(2+\sqrt3)^{2^{p-1}}\ne1.
\]

In a cyclic group of even order (written multiplicatively), there is a unique element of order $2$, namely $-1$.
% so the unique solution to $x^2 = 1$ in $C$ is $x=-1$.
Thus $(2+\sqrt3)^{2^{p-1}} = -1$.
Applying the identity \eqref{eq:alphaformula},
we obtain
\[
T^{\circ (p-1)}(4)=-2.
\]
Since $T(x)=-2$ only at the critical point $x=0$, we conclude that
\[
T^{\circ (p-2)}(4)=0
\quad\text{in }\mathbb F_q,
\]
as desired.

\medskip

($\Leftarrow$) Conversely, suppose
\[
T^{\circ (p-2)}(4)\equiv0\pmod q,
\]
and let $\ell$ be any prime divisor of $q$.

We work in the field $K=\mathbb F_\ell(\sqrt3)$,
which has order equal to either $\ell$ (if $\sqrt3 \in \mathbb F_\ell$) or $\ell^2$ (if not).
% which is either $\mathbb F_\ell$ or $\mathbb F_{\ell^2}$.
% The identity \eqref{eq:alphaformula} holds in $K$.

From the assumption, we obtain
\[
T^{\circ p}(4)=2
\quad\text{and}\quad
T^{\circ (p-1)}(4)=-2
\]
in $K$.
It follows from \eqref{eq:alphaformula} and \eqref{eq:u=1} that
\[
(2+\sqrt3)^{2^p}=1
\quad\text{but}\quad
(2+\sqrt3)^{2^{p-1}}\ne1
\]
in $K^\times$.
Thus $2+\sqrt3$ has order $2^p=q+1$ in $K^\times$, 
so by Lagrange’s theorem, $(q+ 1) \bigm | |K^\times|$.

Since $|K^\times|$ is either $\ell-1$ or $\ell^2-1$,
we conclude that $(q+1)\mid (\ell^2-1)$.
In particular $\ell>\sqrt q$.

Because this holds for every prime divisor $\ell$ of $q$, we conclude that $q$ is prime.
\end{proof}

\subsection*{Conclusion}

Lucas verified the primality of $M_{127}$ by carrying out a prototype of the Lucas–Lehmer test by hand on a $127\times127$ chessboard, encoding squaring in binary and reducing modulo $M_{127}$ step by step.  Today, essentially the same recurrence is executed by computers participating in the Great Internet Mersenne Prime Search, certifying primes with tens of millions of digits.  Although the scale has changed dramatically, the underlying mathematics has not.  The relationship between the squaring map, the Chebyshev polynomial $x^2-2$, and the arithmetic of finite fields remains the engine behind both Lucas’s chessboard computation and the largest known primality proofs of the present day.

\subsection*{Acknowledgments}

We thank Greg Kuperberg for helpful conversations.
We also thank Joe Silverman, Hugh Williams, and the anonymous referee for their feedback on a preliminary version of the manuscript, and ChatGPT 5.4 for assistance with proofreading and polishing the prose.

\end{document}